\documentclass[11pt]{amsart}

\usepackage{amssymb}
\usepackage{epsfig}
\usepackage{amsmath}
\usepackage{color}
\usepackage{mathrsfs}

\usepackage{delarray}
\usepackage[all]{xy}
\usepackage{mathtools}
\usepackage{graphics}
\usepackage{graphicx}
\usepackage{xkeyval}
\usepackage{color}
\usepackage{keyval}
\usepackage{pict2e}
\usepackage{ifthen}
\usepackage{amsmath,amssymb}
\usepackage{amsmath,amscd}
\usepackage{amstext}
\usepackage{amsxtra}
\usepackage{amsthm}
\usepackage{latexsym}
\usepackage{verbatim}
\usepackage{pstricks}
\usepackage[latin1]{inputenc}
\usepackage{multicol}
\usepackage{mathrsfs}
\usepackage[all]{xy}
\usepackage[pdftex]{hyperref}
\usepackage{enumerate}
\DeclareMathAlphabet{\mathcalligra}{T1}{calligra}{m}{n}
\numberwithin{equation}{section}
\theoremstyle{plain}
\newtheorem{theorem}{Theorem}[section]

\newtheorem{lemma}[theorem]{Lemma}

\theoremstyle{definition}

\newtheorem{definition}[theorem]{Definition}
\newtheorem*{condition-S}{Condition}

\theoremstyle{remark}
\newtheorem{remark}[theorem]{Remark}
\newtheorem{example}[theorem]{Example}

\newcommand{\C}{\mathbb{C}}
\newcommand{\f}{\varphi}
\newcommand{\p}{\psi}
\newcommand{\e}{\varepsilon}

\objectmargin{1,5ex}

\begin{document}

\title{Deformations of Legendrian Curves}
\author{Marco Silva Mendes}
\author{Orlando Neto}
\date{\today}

\begin{abstract}
We construct versal and equimultiple versal deformations of the parametrization of a Legendrian curve.
\end{abstract}

\maketitle

\pagenumbering{arabic}




\section{Contact Geometry}\label{S2}
Let $(X,\mathcal{O}_X)$ be a complex manifold of dimension $3$. A differential form $\omega$ of degree $1$ is said to be a \emph{contact form} if $\omega \wedge d\omega$ never vanishes. Let $\omega$ be a contact form. By Darboux's theorem for contact forms there is locally a system of coordinates $(x,y,p)$ such that $\omega=dy-pdx$.
If $\omega$ is a contact form and $f$ is a holomorphic function that never vanishes, $f\omega$ is also a contact form. We say that a locally free subsheaf $\mathcal L$ of $\Omega^1_X$ is a \emph{contact structure} on $X$ if $\mathcal L$ is locally generated by a contact form. If $\mathcal L$ is a contact structure on $X$ the pair $(X,\mathcal L)$ is called a \emph{contact manifold}. Let $(X_1,{\mathcal L}_1)$ and $(X_2,{\mathcal L}_2)$ be contact manifolds. Let $\chi: X_1 \to X_2$ be a holomorphic map. We say that $\chi$ is a \emph{contact transformation} if $\chi^\ast \omega$ is a local generator of ${\mathcal L}_1$ whenever $\omega$ is a local generator of ${\mathcal L}_2$.

 Let $\theta=\xi dx + \eta dy$ denote the canonical $1$-form of $T^\ast \C^2 = \C^2 \times \C^2$. Let $\pi : \mathbb{P}^\ast \C^2=\C^2 \times \mathbb{P}^1 \to \C^2$ be the \emph{projective cotangent bundle} of $\C^2$, where $\pi(x,y; \xi : \eta)=(x,y)$. Let $U\,[V]$ be the open subset of  $\mathbb{P}^\ast \C^2$ defined by $ \eta \neq 0\,[\xi \neq 0]$. Then $\theta / \eta\,[\theta/\xi]$ defines a contact form $dy-pdx\,[dx-qdy]$ on $U\,[V]$, where $p=- \xi / \eta\,[q=- \eta / \xi]$. Moreover, $dy-pdx$ and $dx-qdy$ define a structure of contact manifold on $\mathbb{P}^\ast \C^2$.
 
If $\Phi(x,y)=(a(x,y),b(x,y))$ with $a,b \in \C\{x,y\}$ is an automorphism of $(\C^2, (0,0))$, we associate to $\Phi$ the germ of contact transformation 
\[
\chi: (\mathbb{P}^\ast \C^2,(0,0;0:1)) \to \left(\mathbb{P}^\ast \C^2,(0,0; -\partial_x b(0,0) : \partial_x a(0,0)\right)
\]
 defined by 
 \begin{equation}\label{E:CANONICAL}
 \chi(x,y;\xi : \eta)=\left(a(x,y),b(x,y); \partial_y b \xi - \partial_x b\eta : - \partial_y a\xi + \partial_x a\eta\right).
 \end{equation}
  If $D \Phi_{(0,0)}$ leaves invariant $\{y=0\}$, then $\partial_x b(0,0)=0$, $\partial_x a(0,0) \neq 0 $ and $\chi(0,0;0:1)=(0,0;0:1)$. Moreover,
 \[
 \chi(x,y,p) = \left(a(x,y),b(x,y),  (\partial_y bp + \partial_x b) / (\partial_y ap + \partial_x a)\right).
 \]
 Let $(X,\mathcal L)$ be a contact manifold. A curve $L$ in $X$ is called \emph{Legendrian} if $\omega |_{L}=0$ for each section $\omega$ of $\mathcal L$.

 Let $Z$ be the germ at $(0,0)$ of an irreducible plane curve parametrized by
 \begin{equation}\label{E:PARPLANE}
 \f(t)=(x(t),y(t)).
 \end{equation}
  We define the \emph{conormal}  of $Z$ as the curve parametrized by 
  \begin{equation}\label{E:CONORMAL1}
  \p(t)=(x(t),y(t); -y'(t) : x'(t)).
  \end{equation}
  
  The conormal of $Z$ is the germ of a Legendrian curve of $\mathbb{P}^\ast \C^2$.
  
  We will denote the conormal of $Z$ by    $\mathbb{P}_Z^\ast \C^2$ and the parametrization (\ref{E:CONORMAL1}) by $\mathcal{C}on \, \f $. 
  
  Assume that the tangent cone $C(Z)$ is defined by the equation $ax+by=0$, with $(a,b) \neq (0,0)$. Then $\mathbb{P}_Z^\ast \C^2$ is a germ of a Legendrian curve at $(0,0;a:b)$. 
  
Let $f \in \C\{t\}$. We say the $f$ has order $k$ and write $ord \, f=k$ or  $ord_t \, f=k$ if $f / t^k$ is a unit of $\C\{t\}$. 

\begin{remark}\label{R:POSITION}
Let $Z$ be the plane curve parametrized by (\ref{E:PARPLANE}). Let $L=\mathbb{P}_Z^\ast \C^2$. Then:
\begin{enumerate}[\upshape (i)]
\item $C(Z)=\{y=0\}$ if and only if $ord \, y > ord \, x$.  If $C(Z)=\{y=0\}$, $L$ admits the parametrization
\[
\p(t)= (x(t),y(t), y'(t)/x'(t))
\]
on the chart $(x,y,p)$.
\item $C(Z)=\{y=0\}$ and $C(L)=\{x=y=0\}$ if and only if $ord \, x < ord \, y < 2ord \, x$.
\item $C(Z)=\{y=0\}$ and $ \{x=y=0\} \nsubseteq C(L) \subset \{y=0\}$ if and only if $ord \, y \geq 2ord \, x$.
\item $C(L )=\{y=p=0\}$ if and only if $ord \, y > 2ord \, x$.
\item $mult \, L \leq mult \, Z$. Moreover, $mult \, L=mult \, Z$ if and only if $ord \, y \geq 2ord \, x$.
\end{enumerate}
\end{remark}
 
If $L$ is the germ of a Legendrian curve at $(0,0;a:b)$, $\pi(L)$ is a germ of a plane curve of $(\C^2,(0,0))$. Notice that all branches of $\pi(L)$ have the same tangent cone.

If $Z$ is the germ of a plane curve with irreducible tangent cone, the union $L$ of the conormal of the branches of $Z$ is a germ of a Legendrian curve. We call $L$ the \emph{conormal} of $Z$.

If $C(Z)$ has several components, the union of the conormals of the branches of $Z$ is a union of several germs of Legendrian curves.

If $L$ is a germ of Legendrian curve, $L$ is the conormal of $\pi(L)$.

Consider in the vector space $\C^2$, with coordinates $x,p$, the symplectic form $dp \wedge dx$. We associate to each symplectic linear automorphism
\[
 (p,x) \mapsto (\alpha p + \beta x, \gamma p + \delta x)
 \]
of $\C^2$ the contact transformation
 \begin{equation}\label{E:PARAB}
(x,y,p)=(\gamma p + \delta x,y + \frac{1}{2} \alpha\gamma p^2 + \beta \gamma xp + \frac{1}{2}\beta \delta x^2,\alpha p + \beta x).
\end{equation}
 We call (\ref{E:PARAB}) a \emph{paraboloidal contact transformation}.

 In the case $\alpha=\delta=0$ and $\gamma= - \beta =1$ we get the so called \emph{Legendre} transformation
 \[
 \Psi(x,y,p)=(p,y-px,-x).
 \]
 
 We say that a germ of a Legendrian curve $L$ of $(\mathbb{P}^\ast \C^2,(0,0;a:b))$ is in \emph{generic position} if $C(L) \not\supset \pi^{-1}(0,0)$.

\begin{remark}\label{R: PREPARATION}
Let $L$ be the germ of a Legendrian curve on a contact manifold $(X,\mathcal L)$ at a point $o$. By the Darboux's theorem
for contact forms there is a germ of a contact transformation $\chi:(X,o) \to (U,(0,0,0))$, where $U=\{\eta \neq 0\}$ is the open subset of $\mathbb{P}^\ast \C^2$ considered above. Hence $C(\pi(\chi(L)))=\{y=0\}$. Applying a paraboloidal transformation to $\chi(L)$ we can assume that $\chi(L)$ is in generic position. If $C(L)$ is irreducible, we can assume $C(\chi(L))=\{y=p=0\}$.
\end{remark}

Following the above remark, from now on we will always assume that every Legendrian curve germ is embedded in $(\C^3_{(x,y,p)},\omega)$, where  $\omega=dy-pdx$.

\begin{example}
\begin{enumerate}
\item
The plane curve $Z=\{ y^2-x^3=0\}$ admits a parametrization $\f(t)=(t^2,t^3)$. The conormal $L$ of $Z$ admits the parametrization $\psi(t)=(t^2,t^3,\frac{3}{2}t)$. Hence $C(L)=\pi^{-1}(0,0)$ and $L$ is not in generic position. If $\chi$ is the Legendre transformation, $C(\chi(L))=\{y=p=0\}$ and $L$ is in generic position. Moreover, $\pi(\chi(L))$ is a smooth curve.
\item The plane curve $Z=\{ (y^2-x^3)(y^2-x^5)=0\}$ admits a parametrization given by 
\[
\f_1(t_1)=({t_1}^2,{t_1}^3), \quad
 \f_2(t_2)=({t_2}^2,{t_2}^5).
 \]
  The conormal $L$ of $Z$ admits the parametrization given by 
 \[
  \psi_1(t_1)=({t_1}^2,{t_1}^3,\frac{3}{2}t_1),\quad
   \psi_2(t_2)=({t_2}^2,{t_2}^5,\frac{5}{2}{t_2}^3).
 \]
 Hence $C(L_1)=\pi^{-1}(0,0)$ and $L$ is not in generic position.  If $\chi$ is the paraboloidal contact transformation
\[
\chi:(x,y,p) \mapsto (x+p,y+\frac{1}{2}p^2,p),
\]
then $\chi(L)$ has branches with parametrization given by 
 \begin{align*}
  \chi(\psi_1)(t_1)&=({t_1}^2+\frac{3}{2}t_1,{t_1}^3+\frac{9}{8}{t_1}^2,\frac{3}{2}t_1),\\
   \chi(\psi_2)(t_2)&=({t_2}^2+\frac{5}{2}{t_2}^3,{t_2}^5+\frac{25}{8}{t_2}^6,\frac{5}{2}{t_2}^3).
 \end{align*}
 Then
 \[
 C(\chi(L_1))=\{y=p-x=0\},\quad C(\chi(L_2))=\{y=p=0\}
 \]
 and $L$ is in generic position.
\end{enumerate}
\end{example}


\section{Relative Contact Geometry}\label{S3}

Set $\mathbf{x}=(x_1,\ldots,x_n)$ and $\mathbf{z}=(z_1,\ldots,z_m)$. Let $I$ be an ideal of the ring $\C\{\mathbf{z}\}$. Let $\widetilde{I}$ be the ideal of $\C\{\mathbf{x},\mathbf{z}\}$ generated by $I$.

\begin{lemma}\label{L:IDEAL}
\begin{enumerate}[(a)]
\item Let $f \in C\{\mathbf{x},\mathbf{z}\}$, $f=\sum_\alpha a_\alpha \mathbf{x}^\alpha$ with $a_\alpha \in \C\{\mathbf{z}\}$. Then $f \in \widetilde{I}$ if and only if $a_\alpha \in I$ for each $\alpha$.
\item If $f \in \widetilde{I}$, then $\partial_{x_i} f \in \widetilde{I}$ for $1 \leq i \leq n$.
\item Let $a_1,\ldots,a_{n-1} \in \C\{\mathbf{x},\mathbf{z}\}$. Let $b, \beta_0 \in \widetilde{I}$. Assume that $\partial_{x_n} \beta_0=0$. If $\beta$ is the solution of the Cauchy problem
\begin{equation}\label{E:CAPR}
\partial_{x_n} \beta - \sum_{i=1}^{n-1} a_i \partial _{x_i} \beta =b, \qquad \beta - \beta_0 \in \C\{\mathbf{x},\mathbf{z}\}x_n,
\end{equation}
then $\beta \in \widetilde{I}$.

\end{enumerate}
\end{lemma}
\begin{proof}
There are $g_1,\ldots,g_{\ell} \in \C\{\mathbf{z}\}$ such that $I=(g_1,\ldots,g_{\ell})$. If $a_\alpha \in I$ for each $\alpha$, there are $h_{i,\alpha} \in \C\{\mathbf{z}\}$ such that $a_\alpha=\sum_{i=1}^{\ell} h_{i,\alpha}g_i$. Hence
$
f=\sum_{i=1}^{\ell}(\sum_\alpha h_{i,\alpha} \mathbf{x}^\alpha)g_i \in \widetilde{I}.
$

If $f\in \widetilde{I}$, there are $H_i \in \C\{\mathbf{x},\mathbf{z}\}$ such that $f=\sum_{i=1}^{\ell} H_i g_i$. There are $b_{i,\alpha} \in \C\{\mathbf{z}\}$ such that $H_i = \sum_\alpha b_{i,\alpha} \mathbf{x}^\alpha$. Therefore $a_\alpha = \sum_{i=1}^{\ell} b_{i,\alpha} g_i \in I$.

We can perform a change of variables that rectifies the vector field $\partial_{x_n} - \sum_{i=1}^{n-1} a_i \partial _{x_i}$, reducing the Cauchy problem (\ref{E:CAPR}) to the Cauchy problem
\[
\partial_{x_n} \beta =b, \qquad  \beta - \beta_0 \in \C\{\mathbf{x},\mathbf{z}\}x_n.
\]
Hence statements $(b)$ and $(c)$ follow from $(a)$.
\end{proof}

Let $J$ be an ideal of $\C\{\mathbf{z}\}$ contained in $I$. Let $X,S$ and $T$ be analytic spaces with local rings $\C\{\mathbf{x}\}, \C\{\mathbf{z}\}/I$ and $\C\{\mathbf{z}\}/J$. Hence $X\times S$ and $X\times T$ have local rings $\mathcal{O}:=\C\{\mathbf{x},\mathbf{z}\}/\widetilde{I}$ and $\widetilde{\mathcal{O}}:= \C\{\mathbf{x},\mathbf{z}\}/\widetilde{J}$.
Let $\mathbf{a_1},\ldots,\mathbf{a_{n-1}},\mathbf{b} \in \mathcal{O}$ and $\mathbf{g} \in \mathcal{O}/x_n\mathcal{O}$. Let $a_i, b \in \widetilde{\mathcal{O}}$ and $g \in \widetilde{\mathcal{O}}/x_n\widetilde{\mathcal{O}}$ be representatives of $\mathbf{a_i}, \mathbf{b}$ and $\mathbf{g}$. Consider the Cauchy problems
\begin{equation}\label{E:CKU}
\partial_{x_n} f + \sum_{i=1}^{n-1} a_i \partial _{x_i} f =b, \qquad f + x_n\widetilde{\mathcal{O}}=g
\end{equation}
and
\begin{equation}\label{E:CKD}
\partial_{x_n} \mathbf{f }+ \sum_{i=1}^{n-1} \mathbf{a_i} \partial _{x_i} \mathbf{f} =\mathbf{b}, \qquad \mathbf{f} + x_n\mathcal{O}=\mathbf{g}.
\end{equation}

\begin{theorem}\label{T:CKT}
\begin{enumerate}[(a)]
\item There is one and only one solution of the Cauchy problem (\ref{E:CKU}).
\item If $f$ is a solution of (\ref{E:CKU}), $\mathbf{f}=f + \widetilde{I}$ is a solution of (\ref{E:CKD}).
\item If $\, \mathbf{f}$ is a solution of (\ref{E:CKD}) there is a representative $f$ of $\, \mathbf{f}$ that is a solution of (\ref{E:CKU}).
\end{enumerate}
\end{theorem}

\begin{proof}
By Lemma \ref{L:IDEAL}, $\partial_{x_i} \widetilde{I}=\widetilde{I}$. Hence $(b)$ holds.

Assume $J=(0)$. The existence and uniqueness of the solution of (\ref{E:CKU}) is a special case of the classical Cauchy-Kowalevski Theorem. There is one and only one formal solution of (\ref{E:CKU}). Its convergence follows from the majorant method. 

The existence of a solution of (\ref{E:CKD}) follows from $(b)$.

Let $\mathbf{f_1},\mathbf{f_2}$ be two solutions of (\ref{E:CKD}). Let $f_j$ be a representative of $\mathbf{f_j}$ for $j=1,2$. Then 
$
\partial_{x_n} (f_2 - f_1) + \sum_{i=1}^{n-1} a_i \partial _{x_i} (f_2 - f_1) \in \widetilde{I}
$
and $f_2-f_1 + x_n\widetilde{\mathcal{O}} \in \widetilde{I} + x_n\widetilde{\mathcal{O}}$. By Lemma \ref{L:IDEAL}, $f_2 - f_1 \in \widetilde{I}$. Therefore $\mathbf{f_1}=\mathbf{f_2}$. This ends the proof of statement $(a)$. Statement $(c)$ follows from statements $(a)$ and $(b)$.

\end{proof}

Set $\Omega_{X|S}^1=\bigoplus_{i=1}^n \mathcal{O}dx_i$. We call the elements of $\Omega_{X|S}^1$ \emph{germs of relative differential forms} on $X\times S$. The map $d: \mathcal{O} \to \Omega_{X|S}^1$ given by $df=\sum_{i=1}^n \partial x_i f dx_i$ is called the \emph{relative differential} of $f$.

Assume that $dim\, X=3$ and let $\mathcal{L}$ be a contact structure on $X$. Let $\rho:X \times S\to X$ be the first projection. Let $\omega$ be a generator of $\mathcal{L}$. We will denote by $\mathcal{L}_S$ the sub $\mathcal{O}$-module of $\Omega_{X|S}^1$ generated by $\rho^\ast \omega$. We call $\mathcal{L}_S$ a \emph{relative contact structure} of $X\times S$. We call $(X\times S, \mathcal{L}_S)$ a relative contact manifold. We say that an isomorphism of analytic spaces 
\begin{equation}\label{E:CHIII1}
\chi : X\times S \to X \times S
\end{equation}
is a \emph{relative contact transformation} if $\chi({\bf 0},s)= ({\bf 0},s)$, $ \chi^\ast \omega \in \mathcal{L}_S$ for each $\omega \in \mathcal{L}_S$ and the diagram
\begin{equation}\label{E:DIAG1}
\xymatrix{
X   \ar@{_{(}->}[d] \ar[r]^{id_X} &X  \ar@{_{(}->}[d]\\
X\times S  \ar[d] \ar[r]^{ \chi} &X\times S \ar[d] \\
S   \ar[r]^{id_S} &S 
}
\end{equation}
commutes.

The demand of the commutativeness of diagram (\ref{E:DIAG1}) is a very restrictive condition but these are the only relative contact transformations we will need. We can and will assume that the local ring of $X$ equals $\C\{x,y,p\}$ and that $\mathcal{L}$ is generated by $dy-pdx$.

Set $\mathcal{O}=\C\{x,y,p,\mathbf{z}\}/\widetilde{I}$ and $\widetilde{\mathcal{O}}=\C\{x,y,p,\mathbf{z}\}/\widetilde{J}$. Let $\mathfrak{m}_X$ be the maximal ideal of $\C\{x,y,p\}$. Let $\mathfrak{m}\,[\widetilde{\mathfrak{m}}]$ be the maximal ideal of $\C\{\mathbf{z}\}/I \, [\C\{\mathbf{z}\}/J]$. Let $\mathfrak{n}\,[\widetilde{\mathfrak{n}}]$ be the ideal of $\mathcal{O}\,[\widetilde{\mathcal{O}}]$ generated by $\mathfrak{m}_X\mathfrak{m}\,[\mathfrak{m}_X\widetilde{\mathfrak{m}}]$.
\begin{remark}\label{R:CHIII2}
If (\ref{E:CHIII1}) is a relative contact transformation, there are $\alpha, \beta,\gamma \in \mathfrak{n}$ such that $\partial_x\beta \in \frak n$ and
\begin{equation}\label{E:CHIII3}
\chi(x,y,p,\mathbf{z})=(x+\alpha,y+\beta,p+\gamma,\mathbf{z}).
\end{equation}
\end{remark}

\begin{theorem}\label{T:RCKT}
\begin{enumerate}[(a)]
\item Let $\chi : X\times S \to X \times S$ be a relative contact transformation. There is $\beta_0 \in \mathfrak{n}$ such that $\partial_p \beta_0=0$, $\partial_x \beta_0 \in \frak n$, $\beta$ is the solution of the Cauchy problem
\begin{equation}\label{E:CAUCHY}
\left( 1+\frac{\partial \alpha}{\partial x} + p\frac{\partial \alpha}{\partial y}\right)\frac{\partial \beta}{\partial p} - p\frac{\partial \alpha}{\partial p}\frac{\partial \beta}{\partial y} - \frac{\partial \alpha}{\partial p}\frac{\partial \beta}{\partial x}=p\frac{\partial \alpha}{\partial p}, \qquad \beta -\beta_0 \in p\mathcal{O}
\end{equation}
and
\begin{equation}\label{E:GAMMA}
\gamma=\left( 1+\frac{\partial \alpha}{\partial x} + p\frac{\partial \alpha}{\partial y}\right)^{-1}\left(\frac{\partial \beta}{\partial x} + p\left(\frac{\partial \beta}{\partial y} - \frac{\partial \alpha}{\partial x} -p\frac{\partial \alpha}{\partial y}\right) \right).
\end{equation}
\item Given $\alpha, \beta_0 \in \mathfrak{n}$ such that $\partial_p \beta_0=0$ and $\partial_x \beta_0 \in \frak n$, there is a unique contact transformation $ \chi$ verifying the conditions of statement $(a)$. We will denote $\chi$ by $\chi_{\alpha,\beta_0}$.
\item Given a relative contact transformation $\widetilde{\chi} : X\times T \to X \times T$ there is one and only one contact transformation $\chi : X\times S \to X \times S$ such that the diagram
\begin{equation}\label{E:DIAG2}
\xymatrix{
X\times S   \ar@{_{(}->}[d]  \ar[r]^{\chi} &X \times S  \ar@{_{(}->}[d] \\
X\times T  \ar[r]^{\widetilde{\chi}} &X\times T 
}
\end{equation}
commutes.
\item Given $\alpha, \beta_0 \in \mathfrak{n}$ and $\widetilde{\alpha}, \widetilde{\beta}_0 \in \widetilde{\mathfrak{n}}$ such that $\partial_p\beta_0=0,\, \partial_p \widetilde{\beta}_0=0$, $\partial_x \beta_0 \in \frak n$, $\partial_x \widetilde{\beta}_0 \in \widetilde{\frak n}$ and $\widetilde{\alpha}, \widetilde{\beta}_0$ are representatives of $\alpha,\beta_0$, set $\chi=\chi_{\alpha,\beta_0}$, $\widetilde{\chi}=\chi_{\widetilde{\alpha}, \widetilde{\beta}_0}$. Then diagram (\ref{E:DIAG2}) commutes.
\end{enumerate}
\end{theorem}

\begin{proof}
Statements $(a)$ and $(b)$ are a relative version of Theorem $3.2$ of \cite{AO}. In \cite{AO} we assume $S=\{0\}$. The proof works as long $S$ is smooth. The proof in the singular case depends on the singular variant of the Cauchy-Kowalevski Theorem introduced in \ref{T:CKT}. Statement $(c)$ follows from statement $(b)$ of Theorem \ref{T:CKT}. Statement $(d)$ follows from statement $(c)$ of Theorem \ref{T:CKT}.
\end{proof}

\begin{remark}
$(1)$ The inclusion $S \hookrightarrow T$ is said to be a \emph{small extension} if the surjective map $\mathcal{O}_T \twoheadrightarrow \mathcal{O}_S$ has one dimensional kernel. If the kernel is generated by $\e$, we have that, as complex vector spaces, $\mathcal{O}_T=\mathcal{O}_S \oplus \e\C$. Every extension of Artinian local rings factors through small extensions.
\end{remark}

\begin{theorem}\label{T:TE}
Let $S \hookrightarrow T$ be a small extension such that
\begin{align*}
\mathcal{O}_S &\cong \C\{\mathbf{z}\},\\
\mathcal{O}_T &\cong \C\{\mathbf{z},\e\}/(\e^2,\e z_1,\ldots \e z_m)=\C\{\mathbf{z}\} \oplus  \C \e.
\end{align*}
Assume $\chi : X\times S \to X \times S$ is a relative contact transformation given at the ring level by
\[
(x,y,p) \mapsto (H_1,H_2,H_3),
\]
$\alpha,\beta_0 \in  \mathfrak{m}_X$, such that  $\partial_p \beta_0=0$ and $\beta_0 \in (x^2,y)$. Then, there are uniquely determined  $\beta, \gamma \in \mathfrak{m}_X$ such that $\beta-\beta_0 \in p \mathcal{O}_X$ and $\widetilde{\chi} :~X\times T \to X \times T$, given by
\[
\widetilde{\chi}(x,y,p,\mathbf{z},\e)=(H_1+\e\alpha,H_2+\e\beta,H_3+\e\gamma,\mathbf{z},\e),
\]
 is a relative contact transformation extending $\chi$ (diagram (\ref{E:DIAG2}) commutes). Moreover, the Cauchy problem (\ref{E:CAUCHY}) for $\widetilde{\chi}$ takes the simplified form
\begin{equation}\label{E:CAUCHYTEBETA}
\frac{\partial \beta}{\partial p}=p\frac{\partial \alpha}{\partial p}, \qquad \beta - \beta_0 \in \C\{x,y,p\}p
\end{equation}
and
\begin{equation}\label{E:CAUCHYTEGAMMA}
\gamma= \frac{\partial \beta}{\partial x} + p(\frac{\partial \beta}{\partial y} - \frac{\partial \alpha}{\partial x}) - p^2\frac{\partial \alpha}{\partial y}.
\end{equation}
\end{theorem}

\begin{proof}
We have that $\widetilde{\chi}$ is a relative contact transformation if and only if there is $f:=f' + \e f'' \in \mathcal{O}_T\{x,y,p\}$ with $f \notin (x,y,p)\mathcal{O}_T\{x,y,p\},\, f' \in \mathcal{O}_S\{x,y,p\}, \, f'' \in \C\{x,y,p\}=\mathcal{O}_X$ such that
\begin{equation}\label{E2}
d(H_2 + \e\beta)-(H_3 + \e\gamma)d(H_1 + \e\alpha)=f(dy-pdx).
\end{equation}
Since $\chi$ is a relative contact transformation we can suppose that
\[
dH_2-H_3dH_1=f'(dy-pdx).
\]
Using the fact that $\e \mathfrak{m}_{\mathcal{O}_T}$ we see that (\ref{E2}) is equivalent to
\[
\begin{cases}
\frac{\partial \beta}{\partial p}=p\frac{\partial \alpha}{\partial p},\\
\gamma= \frac{\partial \beta}{\partial x} + p(\frac{\partial \beta}{\partial y} - \frac{\partial \alpha}{\partial x}) - p^2\frac{\partial \alpha}{\partial y},\\
f''= \frac{\partial \beta}{\partial y} - p\frac{\partial \alpha}{\partial y}.
\end{cases}
\]
As $\beta - \beta_0 \in (p)\C\{x,y,p\}$ we have that $\beta$, and consequently $\gamma$, are completely determined by $\alpha$ and $\beta_0$.
\end{proof}

\begin{remark}\label{R:BETAZERO}
Set $\alpha=\sum_k \alpha_k p^k$, $\beta=\sum_k \beta_kp^k$, $\gamma=\sum_k \gamma_k p^k$, where $\alpha_k, \beta_k, \gamma_k \in \C\{x,y\}$ for each $k \geq 0$ and $ \beta_0 \in (x^2,y)$. Under the assumptions of Theorem~\ref{T:TE},
\begin{enumerate}[\upshape (i)]
\item $\beta_k= \frac{k-1}{k}\alpha_{k-1}, \qquad k \geq1$ .
\item Moreover,
 \[
 \gamma_0= \frac{\partial \beta_0}{\partial x}, \; \gamma_1=\frac{\partial \beta_0}{\partial y} - \frac{\partial \alpha_0}{\partial x}, \; \gamma_k=-\frac{1}{k}\frac{\partial \alpha_{k-1}}{\partial x} - \frac{1}{k-1}\frac{\partial \alpha_{k-2}}{\partial y}, \qquad k\geq 2.
 \]
Since,
\[
\frac{\partial}{\partial y} \gamma_0=\frac{\partial}{\partial x}(\frac{\partial \alpha_0}{\partial x} + \gamma_1),
\]
$\beta_0$ is the solution of the Cauchy problem
\[
\frac{\partial \beta_0}{\partial x}=\gamma_0, \quad \frac{\partial \beta_0}{\partial y} = \frac{\partial \alpha_0}{\partial x} + \gamma_1, \qquad \beta_0 \in (x^2,y).
\]
\end{enumerate}

\end{remark}

\section{Categories of Deformations}\label{S2.2}

A category $\mathfrak{C}$ is called a \emph{groupoid} if all morphisms of $\mathfrak{C}$ are isomorphisms.

Let $p: \mathfrak{F} \to \mathfrak{C}$ be a functor.

Let $S$ be an object of $\mathfrak{C}$. We will denote by $\mathfrak{F}(S)$ the subcategory of $\mathfrak{F}$ given by the following conditions:
\begin{itemize}
\item $\Psi$ is an object of $\mathfrak{F}(S)$ if $p(\Psi)=S$.
\item $\chi$ is a morphism of $\mathfrak{F}(S)$ if $p(\chi)=id_S$.
\end{itemize}

Let $\chi \,[\Psi]$ be a morphism [an object] of $\mathfrak{F}$. Let $f \, [S]$ be a morphism [an object] of $\mathfrak{C}$. We say that $\chi \,[\Psi]$ is a morphism [an object] of $\mathfrak{F}$ over $f \, [S]$ if $p(\chi)=f \, [p(\Psi)=S]$.

A morphism $\chi' : \Psi' \to \Psi$ of $\mathfrak{F}$ over $f: S' \to S$ is called \emph{cartesian} if for each morphism $\chi'' : \Psi'' \to \Psi$ of $\mathfrak{F}$ over $f$ there is exactly one morphism $\chi : \Psi'' \to \Psi'$ over $id_{S'}$ such that $\chi' \circ \chi = \chi''$.

If the morphism $\chi' : \Psi' \to \Psi$ over $f$ is cartesian, $\Psi'$ is well defined up to a unique isomorphism. We will denote $\Psi'$ by $f^\ast \Psi$ or $\Psi \times_S S'$.

We say that $\mathfrak{F}$ is a \emph{fibered category} over $\mathfrak{C}$ if
\begin{enumerate}
\item For each morphism $f:S' \to S$ in $\mathfrak{C}$ and each object $\Psi$ of $\mathfrak{F}$ over $S$ there is a morphism  $\chi': \Psi' \to \Psi$ over $f$ that is cartesian.
\item The composition of cartesian morphisms is cartesian.
\end{enumerate}

A fibered groupoid is a fibered category such that  $\mathfrak{F}(S)$ is a groupoid for each $S \in  \mathfrak{C}$.

\begin{lemma}\label{L:DESCARTES}
If $p: \mathfrak{F} \to \mathfrak{C}$ satisfies $(1)$ and $\mathfrak{F}(S)$ is a groupoid for each object $S$ of $\mathfrak{C}$, then $\mathfrak{F}$ is a fibered groupoid over $\mathfrak{C}$.
\end{lemma}
\begin{proof}
Let $\chi : \Phi \to \Psi$ be an arbitrary morphism of $\mathfrak{F}$. It is enough to show that $\chi$ is cartesian. Set $f=p(\chi)$. Let $\chi' : \Phi' \to \Psi$ be another morphism over $f$. Let $f^\ast \Psi \to \Psi$ be a cartesian morphism over $f$. There are morphisms $\alpha :\Phi' \to f^\ast \Psi$, $\beta : \Phi \to f^\ast \Psi$ such that the solid diagram
\begin{equation}\label{D:DIAGCART}
\xymatrix{
f^\ast \Psi \ar[rd] &\Phi \ar[l]^{\beta} \ar[d]^{\chi} &\Phi' \ar@{.>}[l] \ar@/_1pc/[ll]_{\alpha} \ar[ld]^{\chi'}\\
&\Psi
}
\end{equation}
commutes. Hence $\beta^{-1} \circ \alpha$ is the only morphism over $f$ such that diagram (\ref{D:DIAGCART}) commutes.

\end{proof}

Let $\mathfrak{A}n$ be the category of analytic complex space germs. Let $0$ denote the complex vector space of dimension $0$. Let $p: \mathfrak{F} \to \mathfrak{A}n$ be a fibered category.

\begin{definition}\label{D:VERSALDEF}
Let $T$ be an analytic complex space germ. Let $\psi \, [\Psi]$ be an object of $\mathfrak{F}(0) \, [\mathfrak{F}(T)]$. We say that $\Psi$ is a \emph{versal deformation} of $\psi$ if given
\begin{itemize}
\item a closed embedding $f: T'' \hookrightarrow T'$,
\item a morphism of complex analytic space germs $g:T'' \to T$,
\item an object $\Psi'$ of $\mathfrak{F}(T')$ such that $f^\ast \Psi' \cong g^\ast \Psi$,
\end{itemize}
there is a morphism of complex analytic space germs $h: T' \to T$ such that
\[
h \circ f=g \qquad \text{and} \qquad h^\ast \Psi \cong \Psi'.
\]
If $\Psi$ is versal and for each $\Psi'$ the tangent map $T(h):T_{T'} \to T_T$ is determined by $\Psi'$, $\Psi$ is called a \emph{semiuniversal deformation} of $\psi$.

\end{definition}

Let $T$ be a germ of a complex analytic space. Let $A$ be the local ring of $T$ and let $\mathfrak{m}$ be the maximal ideal of $A$. Let $T_n$ be the complex analytic space with local ring $A /\mathfrak{m}^n$ for each positive integer $n$. The canonical morphisms
\[
A \to A /\mathfrak{m}^n \qquad \text{and} \qquad A /\mathfrak{m}^n \to A /\mathfrak{m}^{n+1}
\]
induce morphisms $\alpha_n: T_n \to T$ and $\beta_n: T_{n+1} \to T_n$.

A morphism $f: T'' \to T'$ induces morphisms $f_n: T''_n \to T'_n$ such that the diagram

\[
\xymatrix{
T''  \ar[r]^{f} &T'\\
T''_n \ar@{^{(}->}[u]^{\alpha''_n}  \ar[r]^{f_n} &T'_n \ar@{^{(}->}[u]_{\alpha'_n}  \\
T''_{n+1} \ar@{^{(}->}[u]^{\beta''_n}   \ar[r]^{f_{n+1}} &T'_{n+1}  \ar@{^{(}->}[u]_{\beta'_n}
}
\]
commutes.

\begin{definition}\label{D:FORMALVERSALDEF}
We will follow the terminology of Definition \ref{D:VERSALDEF}. Let $g_n=g \circ \alpha''_n$. We say that $\Psi$ is a \emph{formally versal deformation} of $\psi$ if there are morphisms $h_n: T'_n \to T$ such that
\[
h_n \circ f_n =g_n, \qquad h_n\circ \beta'_n=h_{n+1} \qquad \text{and} \qquad h_n^\ast \Psi \cong {\alpha'_n}^\ast \Psi'.
\]
If $\Psi$ is formally versal and for each $\Psi'$ the tangent maps $T(h_n): T_{T'_n} \to T_T$ are determined by ${\alpha'_n}^\ast \Psi'$, $\Psi$ is called a \emph{formally semiuniversal deformation} of $\psi$.

\end{definition}

\begin{theorem}[\cite{Flenner}, Theorem 5.2]\label{T:FLENNER}
Let $\mathfrak{F} \to \mathfrak{C}$ be a fibered groupoid. Let $\psi \in \mathfrak{F}(0)$. If there is a versal deformation of $\psi$, every formally versal \em[\em semiuniversal \em]\em deformation of $\psi$ is versal \em[\em semiuniversal \em]\em .
\end{theorem}

Let $Z$ be a curve of $\C^n$ with irreducible components $Z_1,\ldots,Z_r$. Set $\bar{\C}= \bigsqcup_{i=1}^r \bar{C}_i$ where each $\bar{C}_i$ is a copy of $\C$. Let $\varphi_i$ be a parametrization of $Z_i$, $1 \leq i \leq r$. Let $\varphi : \bar{\C} \to \C^n$ be the map such that $\varphi |_{\bar{C}_i} = \varphi_i$, $1 \leq i \leq r$. We call $\varphi$ the \emph{parametrization} of $Z$. 

Let $T$ be an analytic space. A morphism of analytic spaces $\Phi : \bar{\C} \times T \to \C^n \times T$ is called a \emph{deformation of $\varphi$ over $T$} if the diagram
\[
\xymatrix{
\bar{\C}   \ar@{_{(}->}[d] \ar[r]^{\varphi} &\C^n  \ar@{_{(}->}[d]\\ 
\bar{\C}\times T  \ar[d] \ar[r]^{\Phi} &\C^n\times T \ar[d]  \\
T    \ar[r]^{id_T} &T  
}
\]
commutes. The analytic space $T$ is called de \emph{base space} of the deformation.

We will denote by $\Phi_i$ the composition
\[
\bar{C}_i \times T \hookrightarrow \bar{\C} \times T \xrightarrow{\Phi} \C^n \times T \to \C^n, \qquad 1 \leq i \leq r.
\]
The maps $\Phi_i$, $1 \leq i \leq r$, determine $\Phi$.

Let $\Phi$ be a deformation of $\varphi$ over $T$. Let $f: T' \to T$ be a morphism of analytic spaces. We will denote by $f^\ast \Phi$ the deformation of $\varphi$ over $T'$ given by
\[
(f^\ast \Phi)_i= \Phi_i \circ (id_{\bar{C}_i} \times f).
\]
We call $f^\ast \Phi$ the \emph{pullback} of $\Phi$ by $f$.

Let $\Phi':\bar{\C} \times T \to \C^n \times T$ be another deformation of $\varphi$ over $T$. A morphism from $\Phi'$ into $\Phi$ is a pair $(\chi,\xi)$ where $\chi :\C^n \times T \to \C^n \times T$ and $\xi : \bar{\C} \times T \to \bar{\C} \times T$ are isomorphisms of analytic spaces such that the diagram
\[
\xymatrix{
T  &\bar{\C}\times T \ar[l] \ar[r]^{\Phi} &\C^n\times T \ar[r] &T\\
&\bar{\C} \ar@{^{(}->}[u] \ar@{_{(}->}[d] \ar[r]^{\varphi} &\C^n\times\{0\} \ar@{^{(}->}[u] \ar@{_{(}->}[d] \\
T \ar[uu]^{id_T}  &\bar{\C}\times T  \ar[l] \ar@/^2pc/[uu]^{\xi}  \ar[r]^{\Phi'} &\C^n\times T  \ar@/_2pc/[uu]_{\chi} \ar[r] &T \ar[uu]_{id_T}
}
\]
commutes.

Let $\Phi'$ be a deformation of $\varphi$ over $S$ and $f:S \to T$ a morphism of analytic spaces. A \emph{morphism of $\Phi' $ into $ \Phi$ over $f$} is  a morphism from $\Phi'$ into $f^\ast \Phi$. There is a functor $p$ that associates $T$ to a deformation $\Psi$ over $T$ and $f$ to a morphism of deformations over $f$. 

Given $t \in T$ let $Z_t$ be the curve parametrized by the composition
\[
\bar{\C} \times \{t\} \hookrightarrow \bar{\C} \times T \xrightarrow{\Phi} \C^n \times T \to \C^n.
\]
We call $Z_t$ the \emph{fiber of the deformation $\Phi$ at the point $t$}.

Let $\varphi: \bar{\C} \to \C^2$ be the parametrization of a plane curve $Z$. We will denote by $\mathcal{D}ef_{\varphi} \, [\mathcal{D}ef^{\, em}_{\varphi}]$ the \emph{category of deformations [equimultiple deformations] $\Phi$ of (the parametrization $\varphi$ of) the plane curve $Z$}.

Consider in $\C^3$ the contact structure given by the differential form $dy-pdx$. Let $\psi: \bar{\C} \to \C^3$ be the parametrization  of a Legendrian curve $L$. We say that a deformation $\Psi$ of $\psi$ is a \emph{Legendrian deformation of $\psi$} if all of its fibers are Legendrian. We say that $(\chi,\xi)$ is an isomorphism of Legendrian deformations if $\chi : X \times T \to X \times T$ is a relative contact transformation. We will denote by $\widehat{\mathcal{D}ef}_{\psi} \, [\widehat{\mathcal{D}ef}^{\,em}_{\psi}]$ the category of Legendrian [equimultiple Legendrian] deformations of $\psi$. All deformations are assumed to have trivial sections (see \cite{Greuel-book}).

Assume that $\psi=\mathcal{C}on \, \f$ parametrizes a germ of a Legendrian curve $L$, in generic position, in $(\C^3_{(x,y,p)},\omega)$. If $\Phi \in \mathcal{D}ef_{\varphi}$ is given by 
\begin{equation}\label{PHITS}
\Phi_i(t_i,\mathbf{s})=\left(X_i(t_i,\mathbf{s}),Y_i(t_i,\mathbf{s})\right), \qquad 1\leq i \leq r,
\end{equation}
such that $P_i(t_i,\mathbf{s}) := \partial_t Y_i(t_i,\mathbf{s}) / \partial_t X_i(t_i,\mathbf{s}) \in \C\{t_i,\mathbf{s}\}$ for  $1\leq i \leq r$, then
\begin{equation}\label{PSITS}
\Psi_i(t_i,\mathbf{s})=\left(X_i(t_i,\mathbf{s}),Y_i(t_i,\mathbf{s}), P_i(t_i,\mathbf{s})\right).
\end{equation}
defines a deformation $\Psi$ of $\psi$ which we call \emph{conormal of $\Phi$}. Notice that in this case all fibers of $\Phi$ have the same tangent space $\{y=0\}$.
We will denote $\Psi$ by $\mathcal{C}on \, \Phi$. If $\Psi \in \widehat{\mathcal{D}ef}_{\psi}$ is given by (\ref{PSITS}), we call \emph{plane projection of $\Psi$} to the deformation $\Phi$ of $\f$ given by (\ref{PHITS}). We will denote $\Phi$ by $\Psi^\pi$. 

Let us consider the full subcategory $\overset{\to}{\mathcal{D}ef}_{\varphi}$  of the deformations $\Phi \in \mathcal{D}ef^{\, em}_{\varphi}$ such that all fibers of $\Phi$ have the same tangent space $\{y=0\}$. 

 \begin{remark}\label{R:admissible}
 We see immediately  that if $\Phi \in \overset{\to}{\mathcal{D}ef}_{\varphi}$ then $\mathcal{C}on \, \Phi$ exists. However, it should be noted that there are more deformations for which the conormal is defined:
 
 Let $\Phi$ be the deformation of $\f=(t^3,t^{10})$ given by
 \[
 X(t,s)=st+t^3; \quad Y(t,s)=\frac{5}{12}st^8+t^{10}.
 \]
 Then $\mathcal{C}on \, \Phi$ exists, but $\Phi$ is not equimultiple.
 \end{remark}
  We define in this way the functors
\[
\mathcal{C}on:  \overset{\to}{\mathcal{D}ef_{\varphi}} \to  \widehat{\mathcal{D}ef}_{\psi}, \qquad
\pi:  \widehat{\mathcal{D}ef}_{\psi} \to \mathcal{D}ef_{\varphi}.
\]
Notice that the conormal of the plane projection of a Legendrian deformation always exists and we have that $\mathcal{C}on \, (\Psi^\pi)=\Psi$ for each $\Psi \in \widehat{\mathcal{D}ef}_{\psi}$ and $(\mathcal{C}on \, \Phi)^\pi=\Phi$ where $\Phi \in \overset{\to}{\mathcal{D}ef_{\varphi}}$.

Let us denote by $\overset{\twoheadrightarrow}{\mathcal{D}ef}_{\varphi}$ the subcategory of equimultiple deformations $\Phi$ of $\f$ such that all fibres of $\Phi$ have fixed tangent space $\{y=0\}$ with conormal in generic position. Then $\overset{\twoheadrightarrow}{\mathcal{D}ef}_{\varphi} \subset \overset{\to}{\mathcal{D}ef_{\varphi}}$ and if $\Phi \in \overset{\to}{\mathcal{D}ef_{\varphi}}$ is given by \ref{PHITS}, then $\Phi \in \overset{\twoheadrightarrow}{\mathcal{D}ef}_{\varphi}$ iff
\begin{equation}\label{E:COND1}
ord_{t_i} \, Y_i \geq 2ord_{t_i} \, X_i, \qquad 1\leq i \leq r.
\end{equation}
Because we demand that $\Phi$ is equimultiple and all branches have tangent space $\{y=0\}$, \ref{E:COND1} is equivalent to  
\begin{equation}\label{E:COND2}
ord_{t_i} \, Y_i \geq 2m_i, \qquad 1\leq i \leq r,
\end{equation}
where $m_i$ is the multiplicity of the component $Z_i$ of $Z$.

\begin{lemma}\label{L:LEMAEM}
Under the assumptions above, 
\[
\mathcal{C}on \,(\overset{\twoheadrightarrow}{\mathcal{D}ef}_{\varphi}) \subset \widehat{\mathcal{D}ef}^{\,em}_{\psi} \quad \text{and} \quad
(\widehat{\mathcal{D}ef}^{\,em}_{\psi})^\pi \subset \overset{\twoheadrightarrow}{\mathcal{D}ef}_{\varphi}.
\]
\end{lemma}
\begin{proof}
Let $m_i$ be the multiplicity of the component $Z_i$ of $Z$. Let $Z_{i,s} [L_{i,s}]$ be the fiber of $\Phi [\Psi]$ (given by \ref{PHITS} [\ref{PSITS}]) at $s$  . If $\Phi \in \overset{\twoheadrightarrow}{\mathcal{D}ef}_{\varphi}$, $C(L_{i,s}) \not\supset \pi^{-1}(0,0)$ for each $s$, so  $ord_{t_i} \, Y_i \geq 2ord_{t_i} \, X_i=2m_i$. Hence $ord_{t_i} \, P_i \geq m_i$ and $\Psi$ is equimultiple.

If $\Psi \in \widehat{\mathcal{D}ef}^{\,em}_{\psi}$, $ord_{t_i} \, P_i \geq  ord_{t_i} \, X_i$ and we get that $C(L_{i,s}) \not\supset \pi^{-1}(0,0)$ for each $s$. Each component $L_{i,s}$ has multiplicity $m_i$ for each $s$. Hence $mult \, Z_{i,s} \geq m_i$ for each $s$. Since multiplicity is semicontinuous, $mult \, Z_{i,s} = m_i$ for each $s$ and $\Phi$ is equimultiple.
\end{proof}

\begin{lemma}
If $\mathfrak{C}$ is one of the categories $\widehat{\mathcal{D}ef}_{\psi}$, $\widehat{\mathcal{D}ef}^{\,em}_{\psi}$, $p: \mathfrak{C} \to \mathfrak{A}n$ is a fibered groupoid.
\end{lemma}
\begin{proof}
Let $f:S \to T$ be a morphism of $\mathfrak{A}n$. Let $\Psi$ be a deformation over $T$. Then, $(\widetilde{\chi},\widetilde{\xi}): f^\ast \Psi \to \Psi$ is cartesian, with
\[
\widetilde{\xi}(t_i,\mathbf{s})=(t_i,\mathbf{s}), \qquad \widetilde{\chi}(x,y,p,\mathbf{s})=(x,y,p,\mathbf{s}).
\]

This is because if $(\chi,\xi):  \Psi' \to \Psi$ is a morphism over $f$, then by definition of morphism of deformations over different base spaces, $(\chi,\xi)$ is a morphism from $\Psi'$ into $f^\ast \Psi$ over $id_S$.
\end{proof}


\section{Equimultiple Versal Deformations}\label{S2.3}

For Sophus Lie a contact transformation was a transformation that takes curves into curves, instead of points into points. We can recover the initial point of view. Given a plane curve $Z$ at the origin, with tangent cone $\{y=0\}$, and a contact transformation $\chi$ from a neighbourhood of $(0;dy)$ into itself, $\chi$ acts on $Z$ in the following way: $\chi \cdot Z$ is the plane projection of the image by $\chi$ of the conormal of $Z$. We can define in a similar way the action of a relative contact transformation on a deformation of a plane curve $Z$, obtainning another deformation of $Z$.

We say that $\Phi \in \overset{\twoheadrightarrow}{\mathcal{D}ef}_{\varphi} (T)$ is \emph{trivial} (relative to the action of the group of relative contact transformations over $T$) if there is $\chi$ such that $\chi \cdot \Phi:=\pi \circ \chi \circ \mathcal{C}on \, \Phi$ is the constant deformation of $\phi$ over $T$, given by 
\[
(t_i,\mathbf{s}) \mapsto \varphi_i(t_i), \qquad i=1,\ldots,r.
\]
Let $Z$ be the germ of a plane curve parametrized by $\f : \bar{\C} \to \C^2$. In the following we will identify each ideal of $\mathcal{O}_Z$ with its image by $\f^\ast : \mathcal{O}_Z \to \mathcal{O}_{\bar{\C}}$. Hence
\[
\mathcal{O}_Z = \C \left\{ 
\left[
\begin{matrix} 
x_1\\ \vdots \\x_r
\end{matrix}
\right],
\left[
\begin{matrix} 
y_1\\ \vdots \\y_r
\end{matrix}
\right]
\right\}
 \subset \bigoplus_{i=1}^r \C\{t_i\}=\mathcal{O}_{\bar{\C}}.
\]				
Set $\dot{\mathbf{x}}=\left[ \dot{x}_1,\ldots,\dot{x}_r\right]^t$, where $\dot{x}_i$ is the derivative of $x_i$ in order to $t_i$, $1\leq i \leq r$. Let
\[
\dot{\varphi} := \dot{\mathbf{x}}\frac{\partial}{\partial x} +  \dot{\mathbf{y}}\frac{\partial}{\partial y}
\]
be an element of the free $\mathcal{O}_{\bar{\C}}$-module
\begin{equation}\label{E:BIG}
\mathcal{O}_{\bar{\C}}\frac{\partial}{\partial x} \oplus \mathcal{O}_{\bar{\C}}\frac{\partial}{\partial y}.
\end{equation}

Notice that (\ref{E:BIG}) has a structure of $\mathcal{O}_Z$-module induced by $\f^\ast$.

Let $m_i$ be the multiplicity of $Z_i$, $1\leq i \leq r$. Consider the $\mathcal{O}_{\bar{\C}}$-module
\begin{equation}\label{E:SECOND}
\left(\bigoplus_{i=1}^r t^{m_i}_i\C\{t_i\} \frac{\partial}{\partial x}\right) \oplus  \left(\bigoplus_{i=1}^r t^{2m_i}_i\C\{t_i\}\frac{\partial}{\partial y}\right).
\end{equation}
Let $\mathfrak{m}_{\bar{\C}}\dot{\f}$ be the sub $\mathcal{O}_{\bar{\C}}$-module of (\ref{E:SECOND}) generated by
\[
(a_1,\ldots,a_r)\left( \dot{\mathbf{x}}\frac{\partial}{\partial x} +  \dot{\mathbf{y}}\frac{\partial}{\partial y} \right),
\]
where $a_i \in t_i\C\{t_i\},\, 1 \leq i \leq r$. For $i=1,\ldots,r$ set $p_i=\dot{y}_i / \dot{x}_i$. For each $k \geq 0$ set
\[
\mathbf{p}^k = \left[ p_1^k,\ldots,p_r^k\right]^t.
\]
Let $ \widehat{I}$ be the sub $\mathcal{O}_Z$-module of (\ref{E:SECOND}) generated by
\[
\mathbf{p}^k  \frac{\partial}{\partial x} + \frac{k}{k+1}\mathbf{p}^{k+1}\frac{\partial}{\partial y}, \qquad k \geq 1.
\]
Set
\[
\widehat{M}_\f=\frac{ \left(\bigoplus_{i=1}^r t^{m_i}_i\C\{t_i\} \frac{\partial}{\partial x}\right) \oplus  \left(\bigoplus_{i=1}^r t^{2m_i}_i\C\{t_i\}\frac{\partial}{\partial y}\right)}{\mathfrak{m}_{\bar{\C}}\dot{\f} + (x,y)\frac{\partial}{\partial x} \oplus (x^2,y)\frac{\partial}{\partial y} + \widehat{I}}.
\]

Given a category $\mathfrak{C}$ we will denote by $\underline{\mathfrak{C}}$ the set of isomorphism classes of elements of $\mathfrak{C}$.

\begin{theorem}\label{T:DEFINFEQUI}
Let $\psi$ be the parametrization of a germ of a Legendrian curve $L$ of a contact manifold $X$. Let $\chi : X \to \C^3$ be a contact transformation such that $\chi(L)$ is in generic position. Let $\f$ be the plane projection of $\chi \circ \psi$. Then there is a canonical isomorphism
\[
\widehat{\underline{\mathcal{D}ef}}^{\,em}_{\,\psi} (T_\e) \xrightarrow{\sim} \widehat{M}_\f.
\]
\end{theorem}
\begin{proof}
Let $\Psi \in \widehat{\mathcal{D}ef}^{\,em}_{\psi} (T_\e)$. By Lemma~\ref{L:LEMAEM}, $\Psi$ is the conormal of its projection $\Phi \in \overset{\twoheadrightarrow}{\mathcal{D}ef}_{\varphi} (T_\e)$. Moreover, $\Psi$ is given by
\[
\Psi_i(t_i,\e)=(x_i + \e a_i,y_i + \e b_i, p_i + \e c_i),
\]
where $a_i,b_i,c_i \in \C\{t_i\},\, ord\,a_i \geq m_i,\, ord \, b_i \geq 2m_i,\, i=1, \ldots,r$. The deformation $\Psi$ is trivial if and only if $\Phi$ is trivial for the action of the relative contact transformations. $\Phi$ is trivial if and only if there are 
\begin{align*}
\xi_i(t_i)&=\widetilde{t}_i=t_i + \e h_i,\\
\chi(x,y,p,\e)&=(x+\e \alpha,y + \e \beta, p+ \e \gamma,\e),
\end{align*}
such that $\chi$ is a relative contact transformation, $\xi_i$ is an isomorphism, $\alpha,\beta, \gamma \in (x,y,p)\C\{x,y,p\},\, h_i \in t_i\C\{t_i\},\, 1 \leq i \leq r$, and
\begin{align*}
x_i(t_i) + \e a_i(t_i)&=x_i(\widetilde{t}_i)+\e \alpha(x_i(\widetilde{t}_i),y_i(\widetilde{t}_i),p_i(\widetilde{t}_i)),\\
y_i(t_i) + \e b_i(t_i)&=y_i(\widetilde{t}_i)+\e \beta(x_i(\widetilde{t}_i),y_i(\widetilde{t}_i),p_i(\widetilde{t}_i)),
\end{align*}
for $i=1,\ldots,r$. By Taylor's formula $x_i(\widetilde{t}_i)=x_i(t_i) + \e \dot{x_i}(t_i)h_i(t_i), \, y_i(\widetilde{t}_i)=y_i(t_i) + \e \dot{y_i}(t_i)h_i(t_i)$ and 
\begin{align*}
\e \alpha(x_i(\widetilde{t}_i),y_i(\widetilde{t}_i),p_i(\widetilde{t}_i))&=\e \alpha(x_i(t_i),y_i(t_i),p_i(t_i)), \\ 
\e \beta(x_i(\widetilde{t}_i),y_i(\widetilde{t}_i),p_i(\widetilde{t}_i))&=\e \beta(x_i(t_i),y_i(t_i),p_i(t_i)),
\end{align*}
for $i=1,\ldots,r$. Hence $\Phi$ is trivialized by $\chi$ if and only if
\begin{align}
a_i(t_i)&=\dot{x_i}(t_i)h_i(t_i) + \alpha(x_i(t_i),y_i(t_i),p_i(t_i)), \label{E:AAAA} \\
b_i(t_i)&=\dot{y_i}(t_i)h_i(t_i) + \beta(x_i(t_i),y_i(t_i),p_i(t_i)), \label{E:BBBB}
\end{align}
for $i=1,\ldots,r$. By Remark~\ref{R:BETAZERO} (i), (\ref{E:AAAA}) and (\ref{E:BBBB}) are equivalent to the condition
\[
\mathbf{a}\frac{\partial}{\partial x} +  \mathbf{b}\frac{\partial}{\partial y} \in \mathfrak{m}_{\bar{\C}}\dot{\f} + (x,y)\frac{\partial}{\partial x} \oplus (x^2,y)\frac{\partial}{\partial y} + \widehat{I}.
\]
\end{proof}
Set 
\begin{align*}
M_\f&=\frac{ \left(\bigoplus_{i=1}^r t^{m_i}_i\C\{t_i\} \frac{\partial}{\partial x}\right) \oplus  \left(\bigoplus_{i=1}^r t^{m_i}_i\C\{t_i\}\frac{\partial}{\partial y}\right)}{\mathfrak{m}_{\bar{\C}}\dot{\f} + (x,y)\frac{\partial}{\partial x} \oplus (x,y)\frac{\partial}{\partial y}},\\
\overset{\twoheadrightarrow}{M}_\f&=\frac{ \left(\bigoplus_{i=1}^r t^{m_i}_i\C\{t_i\} \frac{\partial}{\partial x}\right) \oplus  \left(\bigoplus_{i=1}^r t^{2m_i}_i\C\{t_i\}\frac{\partial}{\partial y}\right)}{\mathfrak{m}_{\bar{\C}}\dot{\f} + (x,y)\frac{\partial}{\partial x} \oplus (x^2,y)\frac{\partial}{\partial y}}.
\end{align*}
By Proposition~$2.27$ of \cite{Greuel-book},
\[
\underline{\mathcal{D}ef}^{\,em}_{\, \f} (T_\e) \cong M_\f.
\]
A similar argument shows that
\[
 \overset{\twoheadrightarrow}{\mathcal{D}ef}_{\varphi} (T_\e) \cong \overset{\twoheadrightarrow}{M}_\f.
\]
We have linear maps
\begin{equation}\label{E:MANM}
M_\f \overset{\imath}{\hookleftarrow}\overset{\twoheadrightarrow}{M}_\f \twoheadrightarrow \widehat{M}_\f.
\end{equation}

\begin{theorem}[\cite{Greuel-book}, II Theorem $2.38$ $(3)$] \label{T:DEFGREUEL}
Set $k=dim\, M_\f$. Let $ \mathbf{a}^j, \mathbf{b}^j \in \bigoplus_{i=1}^r t^{m_i}_i\C\{t_i\}, \, 1 \leq j \leq k$. If
\begin{equation}\label{E:BASIS}
\mathbf{a}^j \frac{\partial}{\partial x} + \mathbf{b}^j \frac{\partial}{\partial y}= \left[
\begin{matrix} 
a_1^j\\ \vdots \\a_r^j
\end{matrix}
\right]\frac{\partial}{\partial x} +
\left[
\begin{matrix} 
b_1^j\\ \vdots \\b_r^j
\end{matrix}
\right]\frac{\partial}{\partial y},
\end{equation}
$ 1 \leq j \leq k$, represents a basis of $M_\f$, the deformation $\Phi: \bar{\C} \times\C^k \to \C^2 \times \C^k$ given by
\begin{equation}\label{E:DEFXYG}
 X_i(t_i,{\bf s})= x_i(t_i) + \sum_{j=1}^{k} a_i^j(t_i)s_j,\;
 Y_i(t_i,{\bf s})= y_i(t_i) + \sum_{j=1}^{k} b_i^j(t_i)s_j, 
\end{equation}
$i=1,\ldots,r$, is a semiuniversal deformation of $\f$ in $\mathcal{D}ef^{\, em}_{\,\f}$.
\end{theorem}

\begin{lemma}\label{L:LEMAVERSAL}
Set $\overset{\twoheadrightarrow}{k}=dim\, \overset{\twoheadrightarrow}{M}_\f$. Let $ \mathbf{a}^j \in \bigoplus_{i=1}^r t^{m_i}_i\C\{t_i\}, \, \mathbf{b}^j \in \bigoplus_{i=1}^r t^{2m_i}_i\C\{t_i\}$, $1 \leq j \leq \overset{\to}{k}$. If (\ref{E:BASIS}) represents a basis of $\overset{\twoheadrightarrow}{M}_\f$, the deformation $\overset{\twoheadrightarrow}{\Phi}$ given by (\ref{E:DEFXYG}), $1 \leq i \leq r$, is a semiuniversal deformation of $\f$ in $ \overset{\twoheadrightarrow}{\mathcal{D}ef}_{\varphi}$. Moreover, $\mathcal{C}on \, \overset{\twoheadrightarrow}{\Phi}$ is a versal deformation of $\psi$ in $\widehat{\mathcal{D}ef}^{\,em}_{\, \psi}$.
\end{lemma}
\begin{proof}
We will only show the completeness of $ \overset{\twoheadrightarrow}{\Phi}$ and $\mathcal{C}on \, \overset{\twoheadrightarrow}{\Phi}$. Since the linear inclusion map $\imath$ referred in (\ref{E:MANM}) is injective, the deformation $\overset{\twoheadrightarrow}{\Phi}$ is the restriction to  $\overset{\twoheadrightarrow}{M}_\f$ of the deformation $\Phi$ introduced in Theorem~\ref{T:DEFGREUEL}. Let $\Phi_0 \in \overset{\twoheadrightarrow}{\mathcal{D}ef_{\varphi}} (T)$. Since $\Phi_0 \in \mathcal{D}ef^{\, em}_{\, \varphi}(T)$, there is a morphism of analytic spaces $f:T \to M_\f$ such that $\Phi_0 \cong f^\ast \Phi$. Since $\Phi_0 \in \overset{\twoheadrightarrow}{\mathcal{D}ef_{\varphi}} (T)$, $f(T) \subset \overset{\twoheadrightarrow}{M}_\f$. Hence $f^\ast \overset{\twoheadrightarrow}{\Phi} = f^\ast \Phi$.

If $\Psi \in \widehat{\mathcal{D}ef}^{\, em}_{\, \psi}(T)$, $\Psi^\pi \in\overset{\twoheadrightarrow}{\mathcal{D}ef_{\varphi}} (T)$. Hence there is $f:T \to \overset{\twoheadrightarrow}{M}_\f$ such that $\Psi^\pi \cong f^\ast \overset{\twoheadrightarrow}{\Phi}$. Therefore $\Psi=\mathcal{C}on \, \Psi^\pi \cong \mathcal{C}on \, f^\ast \overset{\twoheadrightarrow}{\Phi}=f^\ast \mathcal{C}on \, \overset{\twoheadrightarrow}{\Phi}$.
\end{proof}

\begin{theorem}\label{T:VERSALEQUI}
Let $ \mathbf{a}^j \in \bigoplus_{i=1}^r t^{m_i}_i\C\{t_i\}, \, \mathbf{b}^j \in \bigoplus_{i=1}^r t^{2m_i}_i\C\{t_i\}$, $1\leq j \leq \ell$. Assume that (\ref{E:BASIS}) represents a basis \em[\em a system of generators \em]\em of $\widehat{M}_\f$. Let $\Phi$ be the deformation given by (\ref{E:DEFXYG}), $1 \leq i \leq r$. Then $\mathcal{C}on \, \Phi$ is a semiuniversal \em[\em versal \em]\em deformation of $\psi$ in $\widehat{\mathcal{D}ef}^{\, em}_{\, \psi}$.

\end{theorem}
\begin{proof}
By Theorem~\ref{T:FLENNER} and Lemma~\ref{L:LEMAVERSAL} it is enough to show that $\mathcal{C}on \, \Phi$ is formally semiuniversal [versal].

Let $\imath: T' \hookrightarrow T$ be a small extension. Let $\Psi \in \widehat{\mathcal{D}ef}^{\, em}_{\, \psi}(T)$. Set $\Psi'= \imath^\ast \Psi$. Let $\eta': T' \to \C^\ell$ be a morphism of complex analytic spaces. Assume that $(\chi',\xi')$ define an isomorphism
\[
\eta'^\ast \mathcal{C}on \, \Phi \cong \Psi'.
\]
We need to find $\eta: T \to \C^\ell$ and $\chi,\xi$ such that $\eta' =\eta \circ \imath$ and $\chi,\xi$ define an isomorphism 
\[
\eta^\ast \mathcal{C}on \, \Phi \cong \Psi
\]
that extends $(\chi',\xi')$.
Let $A\,[A']$ be the local ring of $T \, [T']$. Let $\delta$ be the generator of $Ker(A \twoheadrightarrow A')$. We can assume $A' \cong \C\{\mathbf{z}\}/I$, where $\mathbf{z}=(z_1,\ldots,z_m)$. Set
\[
\widetilde{A}'=\C\{\mathbf{z}\} \quad \text{and} \quad \widetilde{A}=\C\{\mathbf{z},\e\}/(\e^2,\e z_1,\ldots,\e z_m).
\]
Let $\mathfrak{m}_A$ be the maximal ideal of $A$. Since $\mathfrak{m}_A \delta=0$ and $\delta \in \mathfrak{m}_A$, there is a morphism of local analytic algebras from $\widetilde{A}$ onto $A$ that takes $\e$ into $\delta$ such that the diagram
\begin{equation}
\xymatrix{
\widetilde{A}   \ar[d] \ar[r] &\widetilde{A}'  \ar[d]\\
A   \ar[r] &A' 
}
\end{equation}
commutes. Assume $\widetilde{T} \,[\widetilde{T}']$ has local ring $\widetilde{A} \,[\widetilde{A}']$. We also denote by $\imath$ the morphism $\widetilde{T}' \hookrightarrow \widetilde{T}$. We denote by $\kappa$ the morphisms $T \hookrightarrow \widetilde{T} $ and $T' \hookrightarrow \widetilde{T}'$.  Let $\widetilde{\Psi} \in \widehat{\mathcal{D}ef}^{\, em}_{\, \psi}(\widetilde{T})$ be a lifting of $\Psi$. 

We fix a linear map $\sigma: A' \hookrightarrow \widetilde{A}'$ such that $\kappa^\ast \sigma=id_{A'}$. Set $\widetilde{\chi}'=\chi_{\sigma(\alpha), \sigma(\beta_0)}$, where  $\chi'=\chi_{\alpha,\beta_0}$. Define $\widetilde{\eta}'$ by  $\widetilde{\eta}'^\ast s_i=\sigma(\eta'^\ast s_i)$, $i=1,\ldots, l$. Let $\widetilde{\xi}'$ be the lifting of $\xi'$ determined by $\sigma$. Then
\[
\widetilde{\Psi}':= \widetilde{\chi}'^{-1} \circ \widetilde{\eta}'^\ast \mathcal{C}on \, \Phi \circ \widetilde{\xi}'^{-1}
\]
is a lifting of $\Psi'$ and
\begin{equation}\label{E:PSITIL'}
\widetilde{\chi}' \circ \widetilde{\Psi}' \circ \widetilde{\xi}'= \widetilde{\eta}'^\ast \mathcal{C}on \, \Phi.
\end{equation}
By Theorem~\ref{T:RCKT} it is enough to find liftings $\widetilde{\chi},\widetilde{\xi},\widetilde{\eta}$ of $\widetilde{\chi}',\widetilde{\xi}',\widetilde{\eta}'$ such that
\[
\widetilde{\chi} \cdot \widetilde{\Psi}^\pi \circ \widetilde{\xi} = \widetilde{\eta}^\ast \Phi
\]
in order to prove the theorem.\\

Consider the following commutative diagram
\[
\xymatrix{
\bar{\C} \times \widetilde{T}' \; \ar[d]^{\widetilde{\Psi}'}  \ar@{^{(}->}[r] &\bar{\C}\times \widetilde{T} \; \ar[d]^{\widetilde{\Psi}} \ar@{.>}[r] &\bar{\C}\times \C^\ell \ar[d]^{\mathcal{C}on \, \Phi}\\
\C^3 \times \widetilde{T}' \; \ar[d]^{pr} \ar@{^{(}->}[r] &\C^3 \times \widetilde{T} \; \ar[d]^{pr} \ar@{.>}[r] &\C^3 \times \C^\ell \ar[d]\\
\widetilde{T}' \; \ar@{^{(}->}[r] \ar@/_2pc/[rr]^{\widetilde{\eta}'} &\widetilde{T} \; \ar@{.>}[r]^{\widetilde{\eta}} &\C^\ell.
}
\]
If $\mathcal{C}on \, \Phi$ is given by 
\[
X_i(t_i,{\bf s}),\; Y_i(t_i,{\bf s}),\; P_i(t_i,{\bf s}) \; \in \C\{\mathbf{s},t_i\},
\]
then $\widetilde{\eta}'^{\ast} \, \mathcal{C}on \, \Phi$ is given by 
\[
X_i(t_i,\widetilde{\eta}'({\bf z})),\; Y_i(t_i,\widetilde{\eta}'({\bf z})),\; P_i(t_i,\widetilde{\eta}'({\bf z}))\; \in \widetilde{A}'\{t_i\}=\C\{\mathbf{z},t_i\}
\]
for $i=1,\ldots,r$.
Suppose that $\widetilde{\Psi}'$ is given by 
\[
U'_i(t_i,{\bf z}),\; V'_i(t_i,{\bf z}),\; W'_i(t_i,{\bf z})\; \in \C\{\mathbf{z},t_i\}.
\]
Then, $\widetilde{\Psi}$ must be given by
\[
U_i=U'_i + \e u_i, \; V_i=V'_i + \e v_i, \; W_i=W'_i + \e w_i \; \in \widetilde{A}\{t_i\}= \C\{\mathbf{z},t_i\} \oplus \e\C\{t_i\}
\]
with $u_i, v_i, w_i \in \C\{t_i\}$ and $i=1,\ldots,r$. By definition of deformation we have that, for each $i$,
\[
(U_i,V_i,W_i)=(x_i(t_i),y_i(t_i),p_i(t_i)) \; mod \; \mathfrak{m}_{\widetilde{A}}.
\]
Suppose  $\widetilde{\eta}':\widetilde{T}' \to \C^\ell$ is given by $(\widetilde{\eta}'_1,\ldots,\widetilde{\eta}'_\ell)$, with $\widetilde{\eta}'_i \in \C\{{\bf z}\}$. Then $\widetilde{\eta}$ must be given by $\widetilde{\eta}=\widetilde{\eta}' + \e\widetilde{\eta}^0$ for some $\widetilde{\eta}^0 = (\widetilde{\eta}^0_1,\ldots,\widetilde{\eta}^0_\ell) \in \C^\ell$. Suppose
that $\tilde{\chi}': \C^3\times \widetilde{T}' \to \C^3\times \widetilde{T'}$ is given at the ring level by
\[
(x,y,p) \mapsto (H'_1,H'_2,H'_3),
\]
such that $H'=id \; mod \; \mathfrak{m}_{\widetilde{A}'}$ with $H'_i \in (x,y,p)A'\{x,y,p\}$. Let
the automorphism $\widetilde{\xi}': \bar{\C}\times \widetilde{T}' \to \bar{\C}\times \widetilde{T}'$ be given at the ring level by
\[
t_i \mapsto h'_i
\]such that $h'=id \; mod \; \mathfrak{m}_{\widetilde{A}'}$ with $h'_i \in (t_i)\C\{\mathbf{z},t_i\}$.

Then, from $\ref{E:PSITIL'}$ follows that
\begin{align}\label{E3}
\nonumber X_i(t_i,\widetilde{\eta}') &= H'_1(U'_i(h'_i),V'_i(h'_i),W'_i(h'_i)),\\
Y_i(t_i,\widetilde{\eta}') &= H'_2(U'_i(h'_i),V'_i(h'_i),W'_i(h'_i)),\\
\nonumber P_i(t_i,\widetilde{\eta}') &= H'_3(U'_i(h'_i),V'_i(h'_i),W'_i(h'_i)).
\end{align}
Now, $\widetilde{\eta}'$ must be extended to $\widetilde{\eta}$ such that the first two previous equations extend as well. That is, we must have
\begin{align}\label{E4}
X_i(t_i,\widetilde{\eta}) &= (H'_1 + \e\alpha)(U_i(h'_i+\e h^0_i), V_i(h'_i+\e h^0_i), W_i(h'_i+\e h^0_i),\\
Y_i(t_i,\widetilde{\eta}) &= (H'_2 + \e\beta)(U_i(h'_i+\e h^0_i), V_i(h'_i+\e h^0_i), W_i(h'_i+\e h^0_i). \notag
 \end{align}
 with $\alpha,\beta \in (x,y,p)\C\{x,y,p\}$, $h^0_i \in (t_i)\C\{t_i\}$ such that 
 \[
(x,y,p) \mapsto (H'_1 + \e \alpha,H'_2 + \e \beta,H'_3 + \e \gamma)
\]
gives a relative contact transformation over $ \widetilde{T}$ for some $\gamma \in (x,y,p)\C\{x,y,p\}$. The existence of this extended relative contact tranformation is guaranteed by Theorem \ref{T:TE}. Moreover, again by Theorem \ref{T:TE} this extension depends only on the choices of $\alpha$ and $\beta_0$.
 So, we need only to find $\alpha$, $\beta_0$, $\widetilde{\eta}^0$ and $h^0_i$ such that (\ref{E4}) holds. Using Taylor's formula and $\e^2=0$ we see that
\begin{align}\label{E5}
\nonumber X_i(t_i, \widetilde{\eta}' + \e \widetilde{\eta}^0) &= X_i(t_i,\widetilde{\eta}') + \e \sum_{j=1}^{\ell}  \frac{\partial X_i}{\partial s_j} (t_i, \widetilde{\eta}')\widetilde{\eta}^0_j\\
(\e  \mathfrak{m}_{\widetilde{A}}=0) \qquad &= X_i(t_i,\widetilde{\eta}') + \e \sum_{j=1}^{\ell}  \frac{\partial X_i}{\partial s_j} (t_i, 0)\widetilde{\eta}^0_j,\\
\nonumber Y_i(t_i, \widetilde{\eta}' + \e \widetilde{\eta}^0) &= Y_i(t_i,\widetilde{\eta}') + \e \sum_{j=1}^{\ell}  \frac{\partial Y_i}{\partial s_j} (t_i, 0)\widetilde{\eta}^0_j.
\end{align}
Again by Taylor's formula and noticing that $\e \mathfrak{m}_{\widetilde{A}}=0$, $\e \mathfrak{m}_{\widetilde{A}'}=0$ in $\widetilde{A}$, $h'=id \; mod \; \mathfrak{m}_{\widetilde{A}'}$ and $(U_i,V_i)=(x_i(t_i),y_i(t_i)) \; mod \; \mathfrak{m}_{\widetilde{A}}$ we see that
\begin{align}\label{E6}
\nonumber U_i(h'_i + \e h^0_i) &= U_i(h'_i) + \e \dot{U_i}(h'_i)h^0_i\\
&= U'_i(h'_i) + \e (\dot{x_i}h_i^0 + u_i),\\
\nonumber V_i(h'_i + \e h^0_i) &= V'_i(h'_i) + \e (\dot{y_i}h_i^0 + v_i).
\end{align}
Now, $H' = id \; mod \; \mathfrak{m}_{\widetilde{A}'}$, so
\[
\frac{\partial H'_1}{\partial x} = 1 \; mod \; \mathfrak{m}_{\widetilde{A}'}, \qquad \frac{\partial H'_1}{\partial y}, \frac{\partial H'_1}{\partial p} \in  \mathfrak{m}_{\widetilde{A}'} \widetilde{A}'\{x,y,p\}.
\]
In particular,
\[
\e \frac{\partial H'_1}{\partial y}= \e \frac{\partial H'_1}{\partial p}=0.
\]
By this and arguing as in (\ref{E5}) and (\ref{E6}) we see that
\begin{align*}
 \nonumber &(H'_1 + \e\alpha)( U'_i(h'_i) + \e (\dot{x_i}h_i^0 + u_i), V'_i(h'_i) + \e (\dot{y_i}h_i^0 + v_i), W'_i(h'_i) + \e (\dot{p_i}h_i^0 + w_i))\\
\nonumber &= H'_1(U'_i(h'_i),V'_i(h'_i),W'_i(h'_i)) + \e(\alpha(U'_i(h'_i),V'_i(h'_i),W'_i(h'_i)) + 1(\dot{x_i}h_i^0 + u_i))\\
&= H'_1(U'_i(h'_i),V'_i(h'_i),W'_i(h'_i)) + \e( \alpha (x_i,y_i,p_i) + \dot{x_i}h_i^0 + u_i),\\
&(H'_2 + \e\beta)( U'_i(h'_i) + \e (\dot{x_i}h_i^0 + u_i), V'_i(h'_i) + \e (\dot{y_i}h_i^0 + v_i), W'_i(h'_i) + \e (\dot{p_i}h_i^0 + w_i))\\
&= H'_2(U'_i(h'_i),V'_i(h'_i),W'_i(h'_i)) + \e( \beta (x_i,y_i,p_i) + \dot{y_i}h_i^0 + v_i)
\end{align*}
Substituting this in (\ref{E4}) and using (\ref{E3}) and (\ref{E5}) we see that we have to find $\widetilde{\eta}^0 = (\widetilde{\eta}^0_1,\ldots,\widetilde{\eta}^0_\ell) \in \C^\ell$, $h^0_i$ such that
\begin{align}\label{E7}
&(u_i(t_i),v_i(t_i)) = \sum_{j=1}^{\ell} \widetilde{\eta}^0_j\left(\frac{\partial X_i}{\partial s_j} (t_i, 0),\frac{\partial Y_i}{\partial s_j} (t_i, 0)\right)-\\
 \nonumber - h^0_i(t_i)&((\dot {x_i}(t_i), \dot{y_i}(t_i)) - (\alpha(x_i(t_i),y_i(t_i),p_i(t_i)),\beta(x_i(t_i),y_i(t_i),p_i(t_i))).
\end{align}
Note that, because of Remark \ref{R:BETAZERO} (i),  $(\alpha(x_i(t_i),y_i(t_i),p_i(t_i)),\beta(x_i(t_i),y_i(t_i),p_i(t_i))) \in \widehat{I}$ for each $i$.  Also note that $\widetilde{\Psi} \in \widehat{\mathcal{D}ef}^{\, em}_{\, \psi}(\widetilde{T})$ means that $u_i \in t_i^{m_i}\C\{t_i\}, v_i \in t_i^{2m_i}\C\{t_i\}$.
Then, if the vectors
\begin{align*}
&\left(\frac{\partial X_1}{\partial s_j} (t_1, 0),\ldots,\frac{\partial X_r}{\partial s_j} (t_r, 0)\right)\frac{\partial}{\partial x} + \left(\frac{\partial Y_1}{\partial s_j} (t_1, 0),\ldots,\frac{\partial Y_r}{\partial s_j}(t_r, 0)\right)\frac{\partial}{\partial y}\\
&=(a_1^j(t_1),\ldots,a_r^j(t_r))\frac{\partial}{\partial x} + (b_1^j(t_1),\ldots,b_r^j(t_r))\frac{\partial}{\partial y}, \qquad j=1,\ldots,\ell
\end{align*}
 form a basis of [generate]  $\widehat{M}_\f$, we can solve (\ref{E7}) with unique $\widetilde{\eta}^0_1, \ldots, \widetilde{\eta}^0_\ell$ [respectively, solve]  for all $i=1,\ldots,r$. This implies that the conormal of $\Phi$ is a formally semiuniversal [respectively, versal]  equimultiple deformation of $\p$ over $\mathbb{C}^{\ell}$.
\end{proof}


\section{Versal Deformations}\label{S7}

Let $f \in \C\{x_1, \ldots,x_n\}$. We will denote by $\int fdx_i$ the solution of the Cauchy problem
\[
\frac{\partial g}{\partial x_i}=f, \qquad g \in (x_i).
\]
Let $\psi$ be a Legendrian curve with parametrization given by
\begin{equation}\label{E:LEGF}
t_i \mapsto (x_i(t_i), y_i(t_i),p_i(t_i)) \qquad i=1,\ldots,r.
\end{equation}
We will call \emph{fake plane projection} of (\ref{E:LEGF}) to the plane curve $\sigma$ with parametrization given by
\begin{equation}\label{E:PLAF}
t_i \mapsto (x_i(t_i),p_i(t_i)) \qquad i=1,\ldots,r.
\end{equation}
We will denote $\sigma$ by $\psi^{\pi_f}$.

Given a plane curve $\sigma$ with parametrization (\ref{E:PLAF}), we will cal \emph{fake conormal} of $\sigma$ to the Legendrian curve $\psi$ with parametrization (\ref{E:LEGF}), where
\[
y_i(t_i)=\int p_i(t_i)\dot{x}_i(t_i)dt_i.
\]

We will denote $\psi$ by $\mathcal{C}on_f \, \sigma$. Applying the construction above to each fibre of a deformation we obtain  functors
\[
\pi_f: \widehat{\mathcal{D}ef}_\psi \to \mathcal{D}ef_\sigma, \qquad \mathcal{C}on_f: \mathcal{D}ef_\sigma \to \widehat{\mathcal{D}ef}_\psi.
\]
Notice that
\begin{equation}\label{E:FFAKE}
\mathcal{C}on_f\,(\Psi^{\pi_f})=\Psi, \qquad (\mathcal{C}on_f \,(\Sigma))^{\pi_f}=\Sigma
\end{equation}
for each $\Psi \in \widehat{\mathcal{D}ef}_\psi$ and each $\Sigma \in \mathcal{D}ef_\sigma$.

Let $\psi$ be the parametrization of a Legendrian curve given by (\ref{E:LEGF}). Let $\sigma$ be the fake plane projection of $\psi$. Set $\dot{\sigma} := \dot{\mathbf{x}}\frac{\partial}{\partial x} +  \dot{\mathbf{p}}\frac{\partial}{\partial p}$.
Let $I^f$ be the linear subspace of
\[
\mathfrak{m}_{\bar{\C}} \frac{\partial}{\partial x} \oplus \mathfrak{m}_{\bar{\C}} \frac{\partial}{\partial p}= \left(\bigoplus_{i=1}^r t_i\C\{t_i\}\frac{\partial}{\partial x}\right) \oplus \left(\bigoplus_{i=1}^r t_i\C\{t_i\}\frac{\partial}{\partial p}\right)
\]
generated by
\[
\alpha_0  \frac{\partial}{\partial x} - \left(  \frac{\partial \alpha_0}{\partial x} +  \frac{\partial \alpha_0}{\partial y}\mathbf{p}\right)\mathbf{p}  \frac{\partial}{\partial p}, \qquad
\left(  \frac{\partial \beta_0}{\partial x} +  \frac{\partial \beta_0}{\partial y}\mathbf{p} \right) \frac{\partial}{\partial p},
\]
and
\[
\alpha_k \mathbf{p}^k \frac{\partial}{\partial x} - \frac{1}{k+1}\left(  \frac{\partial \alpha_k}{\partial x}\mathbf{p}^{k+1} +  \frac{\partial \alpha_k}{\partial y}\mathbf{p}^{k+2} \right) \frac{\partial}{\partial p}, \quad k \geq 1,
\]
where $\; \alpha_k \in (x,y), \beta_0 \in (x^2,y)$ for each $k \geq 0$. Set
\begin{equation*}
M^f_\sigma = \frac{\mathfrak{m}_{\bar{\C}} \frac{\partial}{\partial x} \oplus \mathfrak{m}_{\bar{\C}} \frac{\partial}{\partial p}}{\mathfrak{m}_{\bar{\C}} \dot{\sigma} + I^f}.
\end{equation*}

\begin{theorem}\label{T:FAKETEP}
Assuming the notations above, $\widehat{\underline{\mathcal{D}ef}}_\psi \, (T_\e) \cong M^f_\sigma$.
\end{theorem}
\begin{proof}
Let $\Psi \in \widehat{\mathcal{D}ef}_\psi \, (T_\e)$ be given by
\[
\Psi_i(t_i,\e)=(X_i,Y_i,P_i)=(x_i + \e a_i,y_i + \e b_i, p_i + \e c_i),
\]
where $a_i,b_i,c_i \in \C\{t_i\}t_i$ and  $Y_i= \int P_i \partial_{t_i} X_i dt_i$, $i=1,\ldots,r$. Hence
\[
 b_i= \int (\dot{x}_ic_i + \dot{a}_ip_i)dt_i, \qquad i=1,\ldots,r.
\]
By (\ref{E:FFAKE}) $\Psi$ is trivial if and only if there an isomorphism $\xi : \bar{\C}\times T_\e \to \bar{\C}\times T_\e$ given by
\[
t_i \to \widetilde{t}_i=t_i + \e h_i, \qquad h_i \in \C\{t_i\}t_i, \; i=1,\ldots,r,
\]
and a relative contact transformation $\chi : \C^3 \times T_\e \to \C^3 \times T_\e$ given by
\[
(x,y,p,\e) \mapsto (x+\e \alpha,y + \e \beta, p+ \e \gamma,\e)
\]
such that 
\begin{align*}
X_i&=x_i(\widetilde{t}_i)+\e \alpha(x_i(\widetilde{t}_i),y_i(\widetilde{t}_i),p_i(\widetilde{t}_i)),\\
P_i&=p_i(\widetilde{t}_i)+\e \gamma(x_i(\widetilde{t}_i),y_i(\widetilde{t}_i),p_i(\widetilde{t}_i)),
\end{align*}
$i=1,\ldots,r$.
Following the argument of the proof of Theorem~\ref{T:DEFINFEQUI}, $\Psi^{\pi_f}$ is trivial if and only if
\begin{align*}
a_i(t_i)&=\dot{x_i}(t_i)h_i(t_i) + \alpha(x_i(t_i),y_i(t_i),p_i(t_i)), \\
c_i(t_i)&=\dot{p_i}(t_i)h_i(t_i) + \gamma(x_i(t_i),y_i(t_i),p_i(t_i)),
\end{align*}
$i=1,\ldots,r$. The result follows from Remark~\ref{R:BETAZERO} (ii).
\end{proof}

\begin{lemma}\label{L:FAKEVERSAL}
Let $\psi$ be the parametrization of a Legendrian curve. Let $\Phi$ be the semiuniversal deformation in $\mathcal{D}ef_\sigma$ of the fake plane projection $\sigma$ of $\psi$. Then $\mathcal{C}on_f \, \Phi$ is a versal deformation of $\psi$ in $\widehat{\mathcal{D}ef}_\psi$.
\end{lemma}

\begin{proof}
It follows the argument of Lemma~\ref{L:LEMAVERSAL}.
\end{proof}

\begin{theorem}
Let $ \mathbf{a}^j, \mathbf{c}^j \in \mathfrak{m}_{\bar{\C}}$ such that
\begin{equation}\label{E:GBASIS}
\mathbf{a}^j \frac{\partial}{\partial x} + \mathbf{c}^j \frac{\partial}{\partial p}= \left[
\begin{matrix} 
a_1^j\\ \vdots \\a_r^j
\end{matrix}
\right]\frac{\partial}{\partial x} +
\left[
\begin{matrix} 
c_1^j\\ \vdots \\c_r^j
\end{matrix}
\right]\frac{\partial}{\partial p},
\end{equation}
$ 1 \leq j \leq \ell$, represents a basis \em[\em a system of generators \em]\em of $M^f_\sigma$. Let $\Phi \in \mathcal{D}ef_\sigma$ be given by
\begin{equation}
 X_i(t_i,{\bf s})= x_i(t_i) + \sum_{j=1}^{\ell} a_i^j(t_i)s_j,\;
 P_i(t_i,{\bf s})= p_i(t_i) + \sum_{j=1}^{\ell} c_i^j(t_i)s_j, 
\end{equation}
$i=1,\ldots,r$. Then $\mathcal{C}on_f \, \Phi$ is a semiuniversal \em[\em versal \em]\em deformation of $\psi$ in $\widehat{\mathcal{D}ef}_\psi$.
\end{theorem}
\begin{proof}
It follows the argument of Theorem~\ref{T:VERSALEQUI}, using Remark~\ref{R:BETAZERO} (ii).
\end{proof}


\section{Examples}\label{S8}

\begin{example}
Let $\f (t)=(t^3,t^{10})$, $\psi(t)=(t^3,t^{10},\frac{10}{3}t^7)$, $\sigma(t)=(t^3,\frac{10}{3}t^7)$. The deformations given by
\begin{flalign*}
\bullet \; X(t,\mathbf{s})&= t^3,  & Y(t,\mathbf{s})&= s_1t^4 + s_2t^5 + s_3t^7 + s_4t^8 +  t^{10} + s_5t^{11} + s_6t^{14};
\end{flalign*}
\vspace{-6mm}
\begin{align*}
 \bullet \; X(t,\mathbf{s})&= s_1t + s_2t^2 +t^3,  & Y(t,\mathbf{s})&= s_3t + s_4t^2 + s_5t^4 + s_6t^5 + s_7t^7 + s_8t^8 + \\ 
&&&+ t^{10} + s_9t^{11} + s_{10}t^{14};
\end{align*}
are respectively
\begin{itemize}
\item an equimultiple semiuniversal deformation;
\item a semiuniversal deformation
\end{itemize}
of $\f$.
The conormal of the deformation given by
\begin{align*}
X(t,\mathbf{s})&= t^3,  & Y(t,\mathbf{s})&= s_1t^7 + s_2t^8 +  t^{10} + s_3t^{11};
\end{align*}
is an equimultiple semiuniversal deformation of  $\psi$.
The fake conormal of the deformation given by
\begin{align*}
X(t,\mathbf{s})&= s_1t + s_2t^2 +t^3, & P(t,\mathbf{s})&= s_3t + s_4t^2 + s_5t^4 + s_6t^5 + \frac{10}{3}t^7 + s_7t^8;
\end{align*}
is a semiuniversal deformation of the fake conormal of $\sigma$.
The conormal of the deformation given by
\begin{align*}
X(t,\mathbf{s})&= s_1t + s_2t^2 +t^3, & Y(t,\mathbf{s})&= \alpha_2t^2 + \alpha_3t^3 + \alpha_4t^4 + \alpha_5t^5 + \alpha_6t^6 +\\
&&&+  \alpha_7t^7 + \alpha_8t^8 + \alpha_9t^9 + \alpha_{10}t^{10} + \alpha_{11}t^{11} ;
\end{align*}
with
 \begin{align*}
 \alpha_2&= \frac{s_1s_3}{2}, & \alpha_3&= \frac{s_1s_4 + 2s_2s_3}{3}, & \alpha_4&= \frac{3s_3 + 2s_2s_4}{4}, \\
\alpha_5&= \frac{3s_4 + s_1s_5}{5}, & \alpha_6&= \frac{2s_2s_5 + s_1s_6}{6},   & \alpha_7&= \frac{3s_5 + 2s_2s_6}{7}, \\
\alpha_8&= \frac{10s_1 + 9s_6}{24}, & \alpha_9&= \frac{3s_1s_7 + 20s_2}{27}, & \alpha_{10}&= 1 + \frac{s_2s_7}{5},\\
 \alpha_{11}&= \frac{3s_7}{11},
\end{align*} 
is a semiuniversal deformation of $\psi$.

\end{example}

\begin{example}
Let $Z=\{(x,y) \in \C^2: (y^2-x^5)(y^2-x^7)=0\}$. Consider the parametrization $\f$ of $Z$ given by 
\[
x_1(t_1)=t_1^2, \; y_1(t_1)=t_1^5 \qquad x_2(t_2)=t_2^2, \; y_2(t_2) = t_2^7.
\]
Let $\sigma$ be the fake projection of the conormal of $\f$ given by
\[
x_1(t_1)=t_1^2, \; p_1(t_1)=\frac{5}{2}t_1^3 \qquad x_2(t_2)=t_2^2, \; p_2(t_2) = \frac{7}{2}t_2^5.
\]
The deformations given by
\begin{align*}
\bullet \; X_1(t_1,\mathbf{s})&= t_1^2,  & Y_1(t_1,\mathbf{s})&= s_1t_1^3 + t_1^5,\\
 X_2(t_2,\mathbf{s})&= t_2^2,  & Y_2(t_2,\mathbf{s})&= s_2t_2^2 + s_3t_2^3 + s_4t_2^4 + s_5t_2^5 + s_6t_2^6 + t_2^7 + \\
 &&&+ s_7t_2^8 +  s_8t_2^{10} + s_9t_2^{12};\\ 
 \bullet \; X_1(t_1,\mathbf{s})&= s_1t_1 + t_1^2,  & Y_1(t_1,\mathbf{s})&= s_3t_1 + s_4t_1^3 + t_1^5,\\
 X_2(t_2,\mathbf{s})&= s_2t_2 + t_2^2,  & Y_2(t_2,\mathbf{s})&= s_5t_2 + s_6t_2^2 + s_7t_2^3 + s_8t_2^4 + s_{9}t_2^5 + s_{10}t_2^6 + \\
 &&&+ t_2^7 + s_{11}t_2^8 +  s_{12}t_2^{10} + s_{13}t_2^{12};
\end{align*}
are respectively
\begin{itemize}
\item an equimultiple semiuniversal deformation;
\item a semiuniversal deformation
\end{itemize}
of $\f$. The conormal of the deformation given by
\begin{align*}
X_1(t_1,\mathbf{s})&= t_1^2,  & Y_1(t_1,\mathbf{s})&= t_1^5,\\
 X_2(t_2,\mathbf{s})&= t_2^2,  & Y_2(t_2,\mathbf{s})&= s_1t_2^4 + s_2t_2^5 + s_3t_2^6 + t_2^7 +  s_4t_2^8;
 \end{align*}
is an equimultiple semiuniversal deformation of the conormal of $\f$.
The fake conormal of the deformation given by
\begin{align*}
X_1(t_1,\mathbf{s})&= s_1t_1 + t_1^2,  & P_1(t_1,\mathbf{s})&=s_3t_1 + \frac{5}{2}t_1^3,\\
 X_2(t_2,\mathbf{s})&=s_2t_2 + t_2^2,  & P_2(t_2,\mathbf{s})&= s_4t_2 + s_5t_2^2 + s_6t_2^3 + s_7t_2^4 +  \frac{7}{2}t_2^5 + s_8t_2^6;
 \end{align*}
 is a semiuniversal deformation of the fake conormal of $\sigma$.
 The conormal of the deformation given by
 \begin{align*}
 X_1(t_1,\mathbf{s})&= s_1t_1 + t_1^2,  & Y_1(t_1,\mathbf{s})&= \alpha_2t_1^2 + \alpha_3t_1^3 + \alpha_4t_1^4 + t_1^5,\\
 X_2(t_2,\mathbf{s})&= s_2t_2 + t_2^2,  & Y_2(t_2,\mathbf{s})&= \beta_2t_2^2 + \beta_3t_2^3 + \beta_4t_2^4 + \beta_5t_2^5 + \beta_6t_2^6 + \\
 &&&+ \beta_7t_2^7 + \beta_8t_2^8;
\end{align*}
with
\begin{align*}
\alpha_2&= \frac{s_1s_3}{2}, & \alpha_3&= \frac{2s_3}{3}, & \alpha_4&= \frac{5s_1}{8},\\
\beta_2&= \frac{s_2s_4}{2}, & \beta_3&= \frac{2s_4 + s_2s_5}{3},   & \beta_4&= \frac{2s_5 + s_2s_6}{4}, \\
\beta_5&= \frac{2s_6 + s_2s_7}{5}, & \beta_6&= \frac{4s_7 + 7s_2}{12}, & \beta_7&= 1 + \frac{s_2s_8}{7},\\
 \beta_8&= \frac{2s_8}{8},
\end{align*}
is a semiuniversal deformation of the conormal of $\f$. 

\end{example}

\end{document}